\documentclass[<opyions>]{elsarticle}
\usepackage{amsfonts}
\usepackage{stmaryrd}
\usepackage{mathrsfs}
\usepackage{amssymb}
\usepackage{latexsym,bm}
\usepackage{amsthm}
\usepackage{yhmath}
\usepackage{booktabs}
\usepackage{multirow}
\usepackage[center]{caption}
\usepackage{enumitem}
\usepackage{cases}
\usepackage{graphics}
\usepackage{subfig}
\usepackage{color}

\newtheorem{remark}{Remark}

\newtheorem{lemma}{Lemma}

\newtheorem{theorem}{Theorem}

\journal{Applied Numerical Mathematics}

\begin{document}
\begin{frontmatter}
\title{An alternating direction implicit spectral method for solving two dimensional multi-term time fractional mixed diffusion and diffusion-wave equations}
\author[csust]{Zeting Liu}
\author[qut]{Fawang Liu\corref{cor1}}
\ead{f.liu@qut.edu.au.}
\cortext[cor1]{Corresponding author}
\author[qut]{Fanhai Zeng}

%
\address[csust]{School of Mathematics and Statistics, Beijing Institute of Technology,
Beijing, 100081, P. R. China.}
\address[qut]{School of Mathematical Sciences, Queensland University of Technology, GPO Box 2434, Brisbane, Qld. 4001, Australia}


\begin{abstract}
In this paper, we consider the initial boundary value problem of the two dimensional multi-term time fractional mixed diffusion and diffusion-wave equations. An alternating direction implicit (ADI) spectral method is developed based on Legendre spectral approximation in space and finite difference discretization in time.
Numerical stability and convergence of the schemes are proved, the optimal error is $O(N^{-r}+\tau^2)$, where $N, \tau, r$ are the polynomial degree, time step size and the regularity of the exact solution, respectively. We also consider the non-smooth solution case by adding some correction terms. Numerical experiments are presented to confirm our theoretical analysis.  These techniques can be used to model diffusion and transport of viscoelastic non-Newtonian fluids.
\end{abstract}
\end{frontmatter}

\section{Introduction}
\setcounter{equation}{0}
In the last few decades, fractional order differential equations have been successfully employed for modeling of many different processes and systems, such as physics, chemistry, engineering, astrophysics, classical mechanics,
quantum mechanics, nuclear physics, hadron spectroscopy, reader can refer to
 \cite{VV,RH,LZL15}.

Compared to single-term time fractional partial differential equations(PDEs),
for example, time frational sub-diffusion or diffusion wave equation which the
fractional order is $0<\alpha<1$ and $1<\alpha<2$, respectively.
The multi-term time fractional PDEs are proposed to improve the modelling
accuracy in depicting the anomalous diffusion process, successfully capturing
power-law frequency dependence \cite{JRM}, adequately modeling various types of
viscoelastic damping \cite{Chen12}, properly simulating the unsteady flow of a fractional Maxwell fluid \cite{CC1,CC2,Liu18}.
Qin et al. \cite{Qin17} developed multi-term time fractional Bloch equations and application in magnetic resonance imaging.
Qin et al. \cite{Qin18} considered two-dimensional multi-term time and space fractional Bloch-Torrey model based on bilinear rectangular finite elements. Fan et al. \cite{Fan18a} derived a unstructured mesh finite element method for the two-dimensional multi-term time-space fractional diffusion-wave equation on an irregular convex domain.
Fan et al. \cite{Fan18b} proposed some novel numerical techniques for an inverse problem of the multi-term time fractional partial differential equation.

In this paper, we consider the following two dimensional multi-term
time fractional  mixed diffusion and diffusion-wave equation:
\begin{eqnarray}\label{y1}
_CD_{0,t}^\alpha u(x,y,t)+\sum_{j=1}^{Q}{a_j} {_C}D_{0,t}^{\alpha_j}u(x,y,t)
=\mu\Delta u(x,y,t)+f(x,y,t),
\end{eqnarray}
subject to initial conditions
\begin{eqnarray}\label{y2}
u(x,y,0)=g_1(x,y),\ \ \partial_tu(x,y,0)=g_2(x,y),
\end{eqnarray}
and boundary conditions
\begin{eqnarray}\label{y3}
u(x,y,t)|_{\partial_\Omega}=0,
\end{eqnarray}
where
$(x,y)\in \Omega,\ \ t\in (0,T),$\ $\Omega=I_x\times I_y=(-1,1)\times (-1,1),\
2>\alpha>\alpha_1>...>\alpha_{Q'}>1=\alpha_{Q'+1}>...>\alpha_Q>0,$
$Q$ is an integer and the fractional derivative $\alpha$ and $\alpha_j$
are defined in Caputo sense, $a_j$ and $\mu$ are positive numbers.

The multi-term time-fractional PDEs have
generated considerable interest both in mathematics and in applications.
Existence, uniqueness and a priori estimates for a class of these equations were obtained
by Luchko \cite{Luchko}  based on an appropriate maximum principle and the Fourier method.
The analytical solutions of the multi-term fractional PDEs have been studied by
many  authors. Schneider \cite{Schneider} considered
the fractional diffusion and wave equations in full spaces and half spaces
and obtained the corresponding Green¡¯s functions in terms of Fox functions.
Daftardar-Gejji et al. \cite{Daftardar-Gejji} obtained the linear and non-linear
diffusion-wave equations of fractional order by Adomian decomposition method.
Jiang et al. \cite{Jiang12} derived the analytical solutions for
the multi-term time-space fractional advection-diffusion equations.
Subsequently, Ding et al. \cite{ding} presented the analytical solutions for the multi-term time-space fractional advection-diffusion equations with mixed boundary conditions. Ming et al. \cite{Ming16} proposed Analytical solutions of multi-term time fractional differential equations and application to unsteady flows of generalized viscoelastic fluid.

There are various numerical methods in the numerical analysis and scientific computing for
 multi-term time fractional equation, for multi-term time fractional diffusion equation,
Liu et al. \cite{Liu13} used the fractional Adams-Bashforth method and  fractional Adams-Moulton method as
predictor and corrector formulas.
Jin et al. \cite{jbt} studied the equation in a bounded convex polyhedral domain with smooth and non-smooth data
by finite element method.
Zheng et al.  \cite{Zheng16} proposed a high order scheme using spectral method both in time and space.
For multi-term time fractional diffusion-wave equation,
Ren et al. \cite{rjc} constructed some efficient numerical schemes to solve one-dimensional and
two-dimensional cases by combining the compact difference approach for the spatial discretisation and an
$L1$ approximation for the multi-term time Caputo fractional derivatives.
Dehghan et al. \cite{md} established two schemes which the time order is $3-\alpha$
 and the space derivative was discretized with a fourth-order compact
finite difference procedure and Galerkin spectral method, respectively.
Salehi  \cite{rs} developed a meshless collocation method to solve the equation in two dimensions,
the moving least squares reproducing kernel particle approximation was employed to construct the
shape functions for spatial approximation and the time accuracy was $O(\tau^{3-\alpha}).$
Feng \cite{Feng} derive two new different finite difference schemes to approximate
the unsteady MHD Couette flow of a generalized Oldroyd-B fluid, one the time accuracy  is $O(\tau)$,
the other is $O(\tau^{\min \{3-\gamma_s, 2-\alpha_q, 2-\beta\}})$ and both second order in space, where
$\gamma_s$ and $\alpha_q$ are the biggest time derivative less than $2$ and $1$ respectively, $\beta$
is the time derivative on the space derivative.
Hao \cite{hzp} designed a compact difference scheme with time fractional order $(0<\alpha<1<\beta<2),$
and proved the first order accuracy in time and fourth order in space, this scheme can be applied to multi-term
time fractional and two dimensional cases. 
\textcolor[rgb]{1.00,0.00,0.00}{The interesting readers can also refer to a review paper \cite{lcp}, in which the numerical methods for multi-term time-fractional differential equations were introduced.}

However, the time order of the papers we mentioned above are all depending the fractional derivatives or less than two.
There are some papers in which second order accuracy in time were obtained, see \cite{zeng,wk,jbt2}.
Recently, Wang \cite{WZB} proposed a class of second order approximations by using Lubich's idea \cite{Lub86},
called weighted and shifted Gr\"unwald difference operators for the Riemann-Liouville fractional derivatives.
In this paper, we use Wang's idea to construct an unified numerical scheme which occupies second order in time and
spectral accuracy in space for the mixed diffusion and diffusion-wave equation in two dimensions.
Chen \cite{chenhu} consider this equation in one dimensional case with variable coefficients, compared to Chen's paper,
our main contributions are highlighted as follows:

$\centerdot$ We establish a proper ADI spectral scheme in order to conveniently
 solve the two dimensional multi-term time fractional equation on computer.

$\centerdot$ We rigorously prove the stability and convergence theorems of the ADI spectral scheme, and we get an optimal
 error estimate $O(\tau^2+N^{-r})$, while $O(\tau^2+N^{1-r})$ in \cite{chenhu}.

 $\centerdot$ We consider the non-smooth solution case for this equation which has practical significance,
 more precisely, we add some correction terms to the
 ADI spectral scheme to achieve high accuracy.

The rest of this paper is organized as follows.
In Section 2, some preliminaries and notations are shown.
In Section 3, we establish a unified ADI numerical
scheme for the two dimensional multi-term time fractional mixed diffusion and diffusion-wave equations.
In Section 4, the stability and convergence of the fully discrete scheme are analysed.
In Section 5, correction terms are added to solve non-smooth solution case.
We do some numerical experiments in Section 6.
Finally, a conclusion is made in Section 7.

\section{Preliminaries and notations}
Let $L^2(\Omega)$, $L^{\infty}(\Omega)$, and $H^m(\Omega)$ be the usual Sobolev spaces
 equipped with norms $\parallel\cdot\parallel$, $\parallel\cdot\parallel_{\infty}$ and $\parallel\cdot\parallel_m$.
The inner product of $L^2(\Omega)$ is denoted by $(\cdot, \cdot)$.
Furthermore,
$$
H_0^1(\Omega)=\{v\in H^1(\Omega)| \, v|_{\partial\Omega}=0\}.
$$
We denote by $L^\infty(0,T;H^m(\Omega))$ the space of the measurable functions
$u:[0,T]\rightarrow H^m(\Omega),$ such that
$$
\|u\|_{L^\infty(H^m)}=\operatorname*{ess\ sup}\limits_{0\leq t\leq T}\|u(t)\|_m<\infty.
$$

Let $N$ be a positive integer. We denote by $P_N(\Omega)$ be the space of all polynomials of degree no greater than $N$.
The approximation space $V_N^0$ is defined as
$$
V_N^0=\big(P_N(I_x)\otimes P_N(I_y)\big)\cap H_0^1(\Omega).
$$
Define the orthogonal projection $\Pi_N^{1,0}: H_0^1(\Omega)\rightarrow V_N^0,$
such that
$$
\big(\nabla (\Pi_N^{1,0}u-u), \nabla \phi\big)=0, \ \ \forall \phi\in V_N^0.
$$
Throughout this paper $c$ is a generic positive constant independent of $N$ .
Now we introduce the property of the projector $\Pi_N^{1,0}.$
\begin{lemma} [see \cite{CY}]\label{touying}
Let $s$ and $r$ be real numbers satisfying $0\leq s\leq r$. Then there
exist a projector $\Pi_N^{1,0}$ and a positive constant $c$ depending
only on $r$ such that for any function $u \in H^s(\Omega)\cap H^r(\Omega),$
the following estimate holds:
$$
\|u-\Pi_N^{1,0}u\|_{H^s(\Omega)}\leq cN^{s-r}\|u\|_{H^r(\Omega)}.
$$
\end{lemma}
We also give the Gronwall inequality.
\begin{lemma} [see \cite{AA}]\label{gn}
Assume that $y_1\geq0$, $h_n,\ \varphi_n$ are non-negative sequences
and $\varphi_n$ satisfies
\begin{numcases}{}
\varphi_0\leq y_0,  \nonumber\\
\varphi_n\leq y_0+\tau\sum_{j=0}^{n-1}h_j\varphi_j,\ \ n\geq 1.\nonumber
\end{numcases}
Then it follows
\begin{equation*}
\varphi_n\leq y_0\exp\Big(\tau\sum_{j=0}^{n-1}h_j\Big),\ \ n\geq 1.
\end{equation*}
\end{lemma}

Now we recall some definitions of fractional calculus.
For a given function $f(t),$\ $\alpha>0$, we denote by
$_{RL}D_{0,t}^{-\alpha} f(t)$ the left side Riemann-Liouville fractional integral of order
$\alpha$ which is defined as \cite{IP}
\begin{equation}
_{RL}D_{0,t}^{-\alpha} f(t)=\frac{1}{\Gamma(\alpha)}\int_0^t (t-s)^{\alpha-1}f(s){\rm d}s,\ \ t>0.
\end{equation}
For $n-1<\alpha<n,$ we denote by $_{RL}D_{0,t}^\alpha f(t)$ the left-sided
Riemann-Liouville fractional derivative of order $\alpha$ which is defined as
\begin{equation}
_{RL}D_{0,t}^\alpha f(t)=\frac{1}{\Gamma(n-\alpha)}\partial_t^n\int_0^t (t-s)^{n-\alpha-1}f(s){\rm d}s,\ \ t>0.
\end{equation}
For $n-1<\alpha<n,$ we denote by $_{C}D_{0,t}^\alpha f(t)$ the left-sided Caputo fractional derivative of order $\alpha$ which is defined as
\begin{equation}
_{C}D_{0,t}^\alpha f(t)=\frac{1}{\Gamma(n-\alpha)}\int_0^t (t-s)^{n-\alpha-1}\partial_t^n f(s){\rm d}s,\ \ t>0.
\end{equation}
We also have the following formula \cite{KD}
\begin{equation}\label{crl}
_{RL}D_{0,t}^{-\alpha} {_C}D_{0,t}^\alpha f(t)=f(t)-\sum_{k=0}^{n-1}\frac{\partial_t^n f(0)}{k!}t^k.
\end{equation}
The property of the fractional derivatives and  integrals is that
for any $\alpha, \beta>0,$ we have
\begin{equation}\label{p1}
_{RL}D_{0,t}^{-\alpha}{_{RL}}D_{0,t}^{\beta} f(t)=_{RL}D_{0,t}^{-\alpha+\beta} f(t)
\end{equation}

\section{An ADI spectral scheme}
For the approximation of the Riemann-Liouville fractional derivative,
 one can continuously extend the solution $u(x,y,t)$ to be zero for $t < 0$ if $u(x,y,0) = 0.$
We use the weighted and shifted Gr\"unwald difference to discretize the
 Riemann-Liouville fractional derivative.
Thus we assume that $u(x,y,0)=0,$
otherwise, we can consider $\tilde u=u-u_0.$

Let $\tau$ be the time step size and $M$ be a positive integer with
$\tau=T/M$ and $t_n=n\tau$ for $n=0,1,...,M.$ For the function
$u(x,y,t),$ denote $u^n=u(\cdot,t_n)$
and
$$
\delta_t u^{k+\frac12}=\frac{u^{k+1}-u^k}{\tau},\ \ \ u^{k+\frac12}=\frac{u^{k+1}+u^k}{2}
$$
Taking the operation ${_{RL}}D_{0,t}^{-\beta}$ on both sides of the equation \eqref{y1},
using the properties \eqref{crl}--\eqref{p1},
then the equation \eqref{y1} is equivalent to the following form
\begin{equation}\label{yy1}
\partial_t u+\sum_{j=1}^Qa_j\, {_{RL}}D_{0,t}^{\beta_j} u
= \mu\, {_{RL}}D_{0,t}^{-\beta}\Delta u+g,
\end{equation}
where $\beta=\alpha-1$, $\beta_j=\alpha_j-\alpha+1$
and
\begin{equation}\label{g}
g=\partial_t u_0+\sum_{j=1}^{Q'}a_j\, {_{RL}}D_{0,t}^{\alpha_j-\alpha} \partial_t u_0
+{_{RL}}D_{0,t}^{-\beta}f.
\end{equation}

We can discrete the operator ${_{RL}}D_{0,t}^{\beta}$ as in  \cite{WZB} (see Lemma 2.2) which are shifted Gr\"unwald approximations,
that is for any $-1\leq \beta\leq 1,$ we have
\begin{equation}\label{daoshu}
{_{RL}}D_{0,t}^\beta u^{k+1}=\tau^{-\beta}\sum_{j=0}^{k+1}\lambda_j^{(\beta)}u^{k+1-j}+O(\tau^2)
\triangleq D_\tau^{\beta, k+1} u+O(\tau^2),
\end{equation}
where
\begin{equation}\label{daoshu1}
\lambda_0^{(\beta)}=\Big(1+\frac{\beta}{2}\Big)g_0^{(\beta)},\ \
\lambda_j^{(\beta)}=\Big(1+\frac{\beta}{2}\Big)g_j^{(\beta)}-\frac{\beta}{2}g_{j-1}^{(\beta)},
\ \ j\geq 1
\end{equation}
and
$g_j^{(\beta)}=(-1)^j\binom{\beta}{j} $ for $j\geq 0.$

Thus we obtain the time discretization for the equation \eqref{y1}:
\begin{eqnarray*}
\begin{aligned}
\delta_t u^{k+\frac12}+\sum_{i=1}^Qa_i\, D_\tau^{\beta_i, k+\frac12} u
=\mu\,D_\tau^{-\beta, k+\frac12}\Delta u
+g^{k+\frac12}+O(\tau^2).
\end{aligned}
\end{eqnarray*}

For space discretization, we use Legendre spectral method in both $x$ and $y$
directions. Then the fully discrete scheme for equation \eqref{yy1} is to find
$u_N^{k+1}\in V_N^0$ such that
\begin{eqnarray}\label{f1}
&&\Big(\delta_t u_N^{k+\frac12}, v\Big)
+\sum_{i=1}^Qa_i\, \Big(D_\tau^{\beta_i, k+\frac12} u_N, v\Big) \nonumber\\
&=&-\mu\,\Big(D_\tau^{-\beta, k+\frac12}\nabla u_N, \nabla v\Big)
+\Big(g^{k+\frac12}, v\Big),\ \forall v\in V_N^0.
\end{eqnarray}

As is known to all that ADI method can
significantly reduce the computation time and storage requirements for problems defined in
two dimensional spatial domain. This advantage motivates us to establish an ADI spectral scheme.

Denoting
$$
p^2=1+\frac12\sum_{i=1}^Qa_i\lambda_0^{(\beta_i)}\tau^{1-\beta_i},\ \
q=\frac{\mu}{2}\lambda_0^{(\beta)}\tau^{1+\beta}.
$$
Then, we get the following ADI spectral scheme:
\begin{eqnarray}\label{f3}
&&\Big(\delta_t u_N^{k+\frac12}, v\Big)
+\sum_{i=1}^Qa_i\, \Big(D_\tau^{\beta_i, k+\frac12} u_N, v\Big)
+\frac{q^2}{ p^2}\Big(\partial_x\partial_y \delta_t u_N^{k+\frac12}, \partial_x\partial_y v\Big)\nonumber\\
&=&-\mu\,\Big(D_\tau^{-\beta, k+\frac12}\nabla u_N, \nabla v\Big)
+\Big(g^{k+\frac12}, v\Big).
\end{eqnarray}

\section{Stability and convergence}
In this section, we prove the stability and convergence for the ADI spectral scheme \eqref{f3}.
\begin{lemma}[see \cite{WZB}]\label{co}
For any positive
integer $k$ and real vector $(v_1,v_2,...,v_k)^T \in \mathbb{R},$  it holds
$$
\sum_{n=0}^{k-1}\Big(\sum_{j=0}^n \lambda_j^{(\beta)}v_{n+1-j}\Big)v_{n+1}\geq 0.
$$
\end{lemma}

It is easy to verify that the inner product form also holds by integrating the former equation
 in $x$ and $y$ direction, respectively, i.e.
\begin{equation}\label{jj1}
\sum_{n=0}^{k-1}\Big(\sum_{j=0}^n \lambda_j^{(\beta)}v_{n+1-j},\, v_{n+1}\Big)\geq 0.
\end{equation}
According to \eqref{jj1} and the assumption $u_N^0=0$, we know that
\begin{equation}\label{jj}
\sum_{n=0}^{k-1}\Big(D_{\tau}^{\beta, n+1}u_N,\, u_N^{n+1}\Big)\geq 0.
\end{equation}
Now we present the stability result for the fully discrete scheme \eqref{f3}.
\begin{theorem}\label{stb}
The ADI spectral scheme \eqref{f3} is stable under the condition $\tau<1$,
it satisfies
$$
\|u_N^n\|^2\leq
\exp(2T)\Big(2\tau \sum_{k=0}^{n-1}\|g^{k+\frac12}\|^2\Big),
$$
where $g$ is defined by the equation \eqref{g}.
\end{theorem}
\begin{proof}
Taking $v=2u_N^{k+\frac12}$ in equation \eqref{f3}, we infer that
\begin{equation}\label{s1}
\begin{aligned}
\frac{1}{\tau}(\|u_N^{k+1}\|^2-\|u_N^k\|^2)
+&2\sum_{i=1}^Qa_i\Big(D_\tau^{\beta_i, k+\frac12}u_N, u_N^{k+\frac12}\Big)
+\frac{q^2}{\tau p^2}\big(\|\partial_x\partial_y u_N^{k+1}\|^2-\|\partial_x\partial_y u_N^k\|^2\big)\\
&=-2\mu  \Big(D_\tau^{\beta, k+\frac12}\nabla u_N, \nabla u_N^{k+\frac12}\Big)
+2\big(g^{k+\frac12}, u_N^{k+\frac12}\big).
\end{aligned}
\end{equation}
Summing up the equation \eqref{f3} for $k$ from $0$ to $n-1$,
using \eqref{jj},
H\"older inequality, Young's inequality,
and noticing that
$u_N^0=0$, we obtain
\begin{equation}\label{s2}
\begin{aligned}
\|u_N^n\|^2+\frac{q^2}{p^2}\|\partial_x\partial_yu_N^{n}\|^2\leq&
2\tau \sum_{k=0}^{n-1}\big\|g^{k+\frac12}\big\|\big\|u_N^{k+\frac12}\big\|\\
\leq&
\frac{\tau}{2}\sum_{k=0}^{n-1}\|u_N^{k+1}\|^2
+\frac{\tau}{2}\sum_{k=0}^{n-1}\|u_N^k\|^2
+\tau \sum_{k=0}^{n-1}\big\|g^{k+\frac12}\big\|^2\\
\leq&
\frac{1}{2}\|u_N^n\|^2
+\tau \sum_{k=0}^{n-1}\|u_N^k\|^2
+\tau \sum_{k=0}^{n-1}\big\|g^{k+\frac12}\big\|^2.
\end{aligned}
\end{equation}
Namely,
$$
\|u_N^n\|^2\leq
2\tau \sum_{k=0}^{n-1}\|u_N^k\|^2
+2\tau \sum_{k=0}^{n-1}\|g^{k+\frac12}\|^2.
$$
Using the Gronwall inequality, we deduce that
\begin{equation}
\|u_N^n\|^2\leq
\exp(2T)\Big(2\tau \sum_{k=0}^{n-1}\|g^{k+\frac12}\|^2\Big).
\end{equation}
\end{proof}

Now we give the error estimate for the ADI spectral scheme \eqref{f3}.
\begin{theorem}\label{conv}
Let $u$ be the exact solution of the equation \eqref{y1}--\eqref{y3} and
$u_N^k$ be the solution of the equation \eqref{f3}. Assume that
$u,{_{RL}}D_{0,t}^{\beta_i} u\in L^\infty(0,T; H^r(\Omega)),$
$\partial_t u\in L^2(0,T; H^r(\Omega)),$
$u_{ttt}\in L^\infty(0,T; L^2(\Omega)),$
$r\geq 4$ and $\tau<1,$
then we have
\begin{eqnarray*}
\begin{aligned}
\|u^k-u_N^k\|^2\leq c(N^{-2r}+\tau^4),
\end{aligned}
\end{eqnarray*}
where $c$ is a positive constant which is independent of $N$ and $\tau$.
\end{theorem}
\begin{proof}
Denote
$$
e^k=u^k-u_N^k=(u^k-\Pi_N^{1,0}u^k)+(\Pi_N^{1,0}u^k-u_N^k)\triangleq\eta_N^k+e_N^k.
$$
Notice that $e_N^0=0.$
From the equation \eqref{yy1} and the fully discrete scheme \eqref{f1},
we obtain the following error equation:
\begin{eqnarray}
\begin{aligned}
\Big(\delta_t e^{k+\frac12}, v\Big)+\sum_{i=1}^Q a_i
\Big(D_\tau^{\beta_i, k+\frac12} e, v\Big)
+&\frac{q^2}{p^2}\Big(\partial_x\partial_y\delta_t e^{k+\frac12}, \partial_x\partial_y v\Big)
=-\mu\,\tau^{\beta}\Big(D_\tau^{\beta, k+\frac12}\nabla e, \nabla v\Big)\\
&+\Big(R_\tau^{k+\frac12}, v\Big)+\frac{q^2}{p^2}\Big(\partial_x\partial_y\delta_tu^{k+\frac12}, \partial_x\partial_y v\Big),
\end{aligned}
\end{eqnarray}
where
$|R_\tau^{k+\frac12}|=\big|\delta_t u^{k+\frac12}-\partial_t u^{k+\frac12}\big|\leq c\tau^2.$

According to  $e^k=\eta_N^k+e_N^k$ and the definition of $\Pi_N^{1,0}$,
we infer that
\begin{eqnarray}
\begin{aligned}
&\Big(\delta_t e_N^{k+\frac12}, v\Big)+\sum_{i=1}^Q a_i
\Big(D_\tau^{\beta_i, k+\frac12} e_N, v\Big)
+\frac{q^2}{p^2}\Big(\partial_x\partial_y\delta_te_N^{k+\frac12}, \partial_x\partial_y v\Big)\\
=&
-\mu\tau^\beta\,\Big(D_\tau^{\beta_i, k+\frac12}\nabla e_N, \nabla v\Big)+\Big(R_\tau^{k+\frac12}, v\Big)-
\sum_{i=1}^Q a_i\Big(D_\tau^{\beta_i, k+\frac12} \eta_N, v\Big)\\
&-\Big(\delta_t \eta_N^{k+\frac12}, v\Big)-\frac{q^2}{p^2}\Big(\partial_x\partial_y\delta_t\eta_N^{k+\frac12}, \partial_x\partial_y v\Big)+\frac{q^2}{p^2}\Big(\partial_x\partial_y\delta_tu^{k+\frac12}, \partial_x\partial_y v\Big).
\end{aligned}
\end{eqnarray}
Summing up the above equation for $k$ from 0 to $n-1$,
then taking $v=2e_N^{k+\frac12}$,
using \eqref{jj},
we have
\begin{eqnarray}\label{w1}
\begin{aligned}
\|e_N^n\|^2
\leq&
2\tau\sum_{k=0}^{n-1}\Big(R_\tau^{k+\frac12}, e_N^{k+\frac12}\Big)
-2\tau \sum_{k=0}^{n-1}\sum_{i=1}^Q a_i\Big(D_\tau^{\beta_i, k+\frac12} \eta_N, e_N^{k+\frac12}\Big)
-2\tau\sum_{k=0}^{n-1}\Big(\delta_t \eta_N^{k+\frac12}, e_N^{k+\frac12}\Big)\\
&-\frac{2\tau q^2}{p^2}\sum_{k=0}^{n-1}\Big(\partial_x\partial_y\delta_t \eta_N^{k+\frac12},
\partial_x\partial_y e_N^{k+\frac12}\Big)
+\frac{2\tau q^2}{p^2}\sum_{k=0}^{n-1}\Big(\partial_x\partial_y\delta_t u^{k+\frac12},
\partial_x\partial_y e_N^{k+\frac12}\Big).
\end{aligned}
\end{eqnarray}
Now we estimate the  terms on the right hand of the inequality \eqref{w1}, respectively.
Using H\"older inequality, Young's inequality, and Lemma \ref{touying},
we get
\begin{eqnarray*}\label{w2}
\begin{aligned}
2\tau\sum_{k=0}^{n-1}\Big(R_\tau^{k+\frac12}, e_N^{k+\frac12}\Big)
\leq c\tau^4+\frac{\tau}{10}\sum_{k=0}^{n-1}\big(\|e_N^{k+1}\|^2+\|e_N^k\|^2\big),
\end{aligned}
\end{eqnarray*}
\begin{eqnarray*}\label{w3}
\begin{aligned}
-2\tau\sum_{k=0}^{n-1}\Big(\delta_t \eta_N^{k+\frac12}, e_N^{k+\frac12}\Big)
=&
-\tau\sum_{k=0}^{n-1}\bigg(\frac{1}{\tau}\int_{t_k}^{t_{k+1}}\partial_t \eta_N{\rm d}t, e_N^{k+1}+e_N^k\bigg)\\
\leq&
cN^{-2r}\int_0^T\|\partial_t u\|^2{\rm d}t
+\frac{\tau}{10}\sum_{k=0}^{n-1}\big(\|e_N^{k+1}\|^2+\|e_N^k\|^2\big),
\end{aligned}
\end{eqnarray*}
\begin{eqnarray*}\label{w4}
\begin{aligned}
-2\tau \sum_{k=0}^{n-1}\sum_{i=1}^Q a_i\Big(D_\tau^{\beta_i, k+\frac12} \eta_N,e_N^{k+\frac12}&\Big)
\leq
c\sum_{i=1}^Q a_i^2\big\|D_\tau^{\beta_i, k+\frac12}\big\|^2
+\frac{\tau}{10} \sum_{k=0}^{n-1}\big(\|e_N^{k+1}\|^2+\|e_N^k\|^2\big)\\
\leq&
cN^{-2r}\sum_{i=1}^Q a_i^2\|{_{RL}}D_{0,t}^{\beta_i} u\|_{L^\infty(H^r)}^2+c\tau^4
+\frac{\tau}{10}\big(\|e_N^{k+1}\|^2+\|e_N^k\|^2\big)
\end{aligned}
\end{eqnarray*}
and
\begin{eqnarray*}\label{w5}
\begin{aligned}
-\frac{2\tau q^2}{p^2}\sum_{k=0}^{n-1}\Big(\partial_x\partial_y\delta_t \eta_N^{k+\frac12},
\partial_x\partial_y &e_N^{k+\frac12}\Big)
=-\frac{2\tau q^2}{p^2}\sum_{k=0}^{n-1}\Big(\partial_x^2\partial_y^2\delta_t \eta_N^{k+\frac12},
e_N^{k+\frac12}\Big)\\
&\leq
cN^{8-2r}\tau^{4+4\beta}\int_0^T\|\partial_t u\|_{H^r}^2{\rm d}t
+\frac{\tau}{10}\sum_{k=0}^{n-1}\big(\|e_N^{k+1}\|^2+\|e_N^k\|^2\big)\\
&\leq
c\tau^4\int_0^T\|\partial_t u\|_{H^r}^2{\rm d}t
+\frac{\tau}{10}\sum_{k=0}^{n-1}\big(\|e_N^{k+1}\|^2+\|e_N^k\|^2\big).
\end{aligned}
\end{eqnarray*}
Similarly, we have
\begin{eqnarray*}\label{w6}
\begin{aligned}
\frac{2\tau q^2}{p^2}\sum_{k=0}^{n-1}\Big(\partial_x\partial_y\delta_t u^{k+\frac12},
\partial_x\partial_y e_N^{k+\frac12}\Big)
\leq
c\tau^4\int_0^T\|\partial_t u\|_{H^4}^2{\rm d}t
+\frac{\tau}{10}\sum_{k=0}^{n-1}\big(\|e_N^{k+1}\|^2+\|e_N^k\|^2\big).
\end{aligned}
\end{eqnarray*}
Substituting all the above estimates into the equation \eqref{w1},
and assuming that $\tau<1$ holds,
we infer that
\begin{eqnarray}\label{w7}
\begin{aligned}
\|e_N^n\|^2\leq c(N^{-2r}+\tau^4)+\frac{2\tau}{5}\sum_{k=0}^{n-1}\|e_N^k\|^2.
\end{aligned}
\end{eqnarray}
Applying Lemma \ref{gn} for inequality \eqref{w7},
we deduce that
\begin{eqnarray*}\label{r}
\begin{aligned}
\|e_N^n\|^2\leq c(N^{-2r}+\tau^4).
\end{aligned}
\end{eqnarray*}
Finally, using the triangular inequality
$\|e^n\|\leq \|e_N^n\|+\|\eta_N^n\|$ and Lemma \ref{touying},
we get the desired result.
\end{proof}

\section{Corrections}
In Section $3$, the second-order  weighted shifted GL formula is applied to discretize
the fractional operators. This formula is not a global second-order method, which
preserves second-order accuracy when $t$ is far from the
origin, but may have very low accuracy near the origin, even for smooth solutions, see  \cite{ZengZK17}.
In order to obtain highly accurate numerical solutions, the correction terms are added when
the fractional operators are discrerized, i.e.,
\begin{equation}\label{daoshu-2}
{_{RL}}D_{0,t}^\beta u^{k+1}\approx
\tau^{-\beta}\sum_{j=0}^{k+1}\lambda_j^{(\beta)}u^{k+1-j}+
\tau^{-\beta}\sum_{j=1}^{m}w_{k,j}^{(\beta)}(u^{j}-u^0)
\triangleq D_\tau^{\beta, k+1,m} u,
\end{equation}
where $\lambda_j^{(\beta)}$ are defined by \eqref{daoshu1} and
$w_{k,j}^{(\beta)}$ are the starting weights that are chosen such that
\begin{equation}\label{starting-weights}
{_{RL}}D_{0,t}^\beta u^{k+1}=
\tau^{-\beta}\sum_{j=0}^{k+1}\lambda_j^{(\beta)}u^{k+1-j}+
\tau^{-\beta}\sum_{j=1}^{m}w_{k,j}^{(\beta)}(u^{j}-u^0)
\end{equation}
for $u=t^{\sigma_j}(1\leq j \leq m)$ and $0<\sigma_j<\sigma_{j+1}$.
In the original scheme, the first-order time derivative is discretized by
$\frac{1}{2}(\partial_tu(t_k)+\partial_tu(t_{k+1}))\approx\delta_t u^{k+\frac12}$,
which yields second-order accuracy when $u(t)$ is smooth.
We modify the central difference for non-smooth solution as follows
\begin{equation}\label{eq:sec5-2}
\delta^{(m)}_t u^{k+\frac12}=\delta_t u^{k+\frac12}
+\frac{1}{\tau}\sum_{j=1}^{m}w_{k,j}(u^k-u^0),
\end{equation}
where the starting weights $w_{k,j}$ are chosen such that
\begin{equation}\label{starting-weights-2}
\frac{1}{2}\big(\partial_tu(t_k)+\partial_tu(t_{k+1})\big)
=\frac{1}{\tau}(u^{k+1}-u^k)+\frac{1}{\tau}\sum_{j=1}^{m}w_{k,j}(u^k-u^0)
\end{equation}
for $u=t^{\sigma_j}(1\leq j \leq m)$ and $0<\sigma_j<\sigma_{j+1}$.

In order to derive the ADI scheme, the perturbation term
$\frac{q^2}{p^2}(\partial_x\partial_y \delta_t u_N^{k+\frac12},\partial_x\partial_y v)$
is added to the non-ADI scheme \eqref{f1} to obtain the ADI scheme \eqref{f3}.
If $u$ is non-smooth with low regularity, then the perturbation term
$\frac{q^2}{p^2}(\partial_x\partial_y \delta_t u_N^{k+\frac12},\partial_x\partial_y v)$  may cause
large errors, which is resolved by the following modified perturbation
\begin{equation}\label{eq:sec5-3}
\frac{q^2}{p^2}\big(\partial_x\partial_y \delta_t u_N^{k+\frac12},\partial_x\partial_y v\big)
+\frac{q^2}{p^2}\frac{1}{\tau}\sum_{j=1}^mW_{k,j}\big(\partial_x\partial_y(u_N^{j}-u_N^0),\partial_x\partial_y v\big),
\end{equation}
where the starting weights are chosen such that
\begin{equation}\label{starting-weights-3}
u^{k+1}-u^k+\sum_{j=1}^mW_{k,j}(u^{j}-u^0)=0
\end{equation}
for $u=t^{\sigma_j}(1\leq j \leq m)$ and $0<\sigma_j<\sigma_{j+1}$.

Combining \eqref{daoshu-2}, \eqref{eq:sec5-2}, and \eqref{eq:sec5-3}, we obtain the modified
time discretization of \eqref{yy1} as follows:
\begin{equation}\label{eq:sec5-4}
\begin{aligned}
&\delta^{(m)}_t u^{k+\frac12}
+\sum_{i=1}^Qa_i\, D_\tau^{\beta_i, k+\frac12,m} u
+\frac{q^2}{ p^2}\partial^2_x\partial^2_y
\left(\delta_t u^{k+\frac12}+\frac{1}{\tau}\sum_{j=1}^mW_{k,j}(u^{j}-u^0)\right)\\
=&-\mu\,D_\tau^{-\beta, k+\frac12,m}\Delta u+g^{k+\frac12}+O(\tau^2),
\end{aligned}
\end{equation}
where $\delta^{(m)}_t$ is defined by  \eqref{eq:sec5-2}, and
\begin{equation}\label{eq:sec5-5}
\begin{aligned}
D_\tau^{\beta, k+\frac12,m} u=&\frac12\left(D_\tau^{\beta, k+1} u  +  D_\tau^{\beta, k} u\right)
+ \frac12\tau^{-\beta}\sum_{j=1}^{m}\left(w_{k+1,j}^{(\beta_i)}+w_{k,j}^{(\beta)}\right)(u^{j}-u^0).
\end{aligned}
\end{equation}

From \eqref{eq:sec5-4}, we derive the following improved ADI spectral method:
find$u_N^{k+1}\in V_N^0$ for $k\geq m$ such that
\begin{eqnarray}\label{ADI-corrections}
\begin{aligned}
&\Big(\delta^{(m)}_t u_N^{k+\frac12}, v\Big)
+\sum_{i=1}^Qa_i\, \Big(D_\tau^{\beta_i, k+\frac12,m} u_N, v\Big)\\
&+\frac{q^2}{ p^2}\Big(\partial_x\partial_y \delta_t u_N^{k+\frac12}, \partial_x\partial_y v\Big)
+\frac{q^2}{ p^2}\frac{1}{\tau}\sum_{j=1}^mW_{k,j}
\left(\partial_x\partial_y(u^{j}-u^0),\partial_x\partial_y v\right)\\
=&-\mu\,\Big(D_\tau^{-\beta, k+\frac12,m}\nabla u_N, \nabla v\Big)
+\Big(g^{k+\frac12}, v\Big).
\end{aligned}
\end{eqnarray}

\begin{remark}
The starting values of
$u_N^j (1\leq j \leq m)$ need to be known to start new ADI method \eqref{ADI-corrections}. This could be done
by applying the original ADI method \eqref{f3} or non-ADI method \eqref{f1} with smaller time stepsize.
Other high-order time-stepping methods can be also applied here to obtain the starting values.
In this work, we apply the   original ADI method \eqref{f3} with smaller time stepsize to derive
these values.
\end{remark}

\begin{remark}
For the modified fully discrete scheme \eqref{ADI-corrections},
the correction terms do not affect the stability, because we just need
to move the correction terms to the right side of the equation \eqref{ADI-corrections}, and these terms are bounded. However,
this modified scheme \eqref{ADI-corrections} has higher accuracy in time than the original scheme \eqref{f3} for non-smooth solutions, which will be verified in the
following numerical simulations. For more detail, we refer readers to \cite{ZengZK17} and the references therein.
\end{remark}
\section{Numerical experiment}
In this section, we carry out numerical experiments by
using the ADI spectral scheme
to illustrate our theoretical results.

\subsection{Numerical implementation}
Let $L_n(x)$ denote Legendre polynomials of degree $n.$
We select the basis functions as follows:
$$
\varphi_i(x)=L_i(x)-L_{i+2}(x),\ \ i= 0, 1,\cdots, N-2,
$$
$$
\psi_j(y)=L_j(y)-L_{j+2}(y),\ \ j= 0, 1,\cdots, N-2.
$$
It is easy to verify that the basis functions satisfy the zero boundary conditions
and we can write
$V_N^0=span\big\{\varphi_i(x)\psi_j(y),\ i,j=0,1,\cdots, N-2\big\}$.
Then we have
$$
u_N^{k+1}=\sum_{i=0}^{N-2}\sum_{j=0}^{N-2}\hat{u}_{ij}^{k+1}\varphi_i(x) \psi_j(y),
$$
where
$\{\hat{u}_{ij}^{k+1}\}_{i,j=0}^{N-2}$
are the frequency coefficients.

We rewrite the equation \eqref{f3} as follows:
\begin{eqnarray*}\label{f4}
\begin{aligned}
p^2\big(u_N^{k+1}, v\big)
+&q\big(\nabla u_N^{k+1}, \nabla v\big)
+\frac{q^2}{p^2}\Big(\partial_x\partial_y u_N^{k+1}, \partial_x\partial_y v\Big)
=\big(u_N^k, v\big)+\tau\big(g^{k+\frac12}, v\big)\\
&+
\frac{q^2}{p^2}\Big(\partial_x\partial_y u_N^k, \partial_x\partial_y v\Big)
-\sum_{i=1}^Q a_i\tau^{1-\beta_i} \bigg(\sum_{j=0}^{k}\big(\lambda_{j+1}^{(\beta_i)}+\lambda_{j}^{(\beta_i)}\big)
u_N^{k-j}, v\bigg)\\
&-\frac{\mu}{2}\tau^{1+\beta}\,\bigg(\sum_{j=0}^{k}(\lambda_{j+1}^{(\beta_i)}+\lambda_{j}^{(\beta_i)})
\nabla u_N^{k-j}, \nabla v\bigg)
\triangleq
F^k(v).\ \ \  \forall v\in V_N^0.
\end{aligned}
\end{eqnarray*}
Then for $l,s=0,1,\cdots, N-2$,
we have
\begin{eqnarray}\label{aeq}
\begin{aligned}
\sum_{i=0}^{N-2}\sum_{j=0}^{N-2}\bigg\{p^2\big(\varphi_i \psi_j, \varphi_l\psi_s\big)
+&q\big(\partial_x\varphi_i\psi_j, \partial_x\varphi_l\psi_s\big)
+q\big(\varphi_i \partial_y\psi_j,\varphi_l\partial_y \psi_s\big)\\
&+
\frac{q^2}{p^2}\Big(\partial_x\varphi_i\partial_y\psi_j, \partial_x\varphi_l\partial_y\psi_s\Big)\bigg\}
\hat{u}_{ij}^{k+1}
=F^k(\varphi_l\psi_s).
\end{aligned}
\end{eqnarray}
Moreover, we infer that
\begin{eqnarray}\label{meq}
\begin{aligned}
\bigg(p^2 M_x \otimes M_y+q\big(S_x\otimes M_y+M_x\otimes S_y\big)+\frac{q^2}{p^2}S_x\otimes S_y\bigg)U^{k+1}=F^k,
\end{aligned}
\end{eqnarray}
where
$$
M_x=\big((\varphi_i, \varphi_j)_{I_x}\big)_{i,j=0}^{N-2},\ \ \
M_y=\big((\psi_i, \psi_j)_{I_y}\big)_{i,j=0}^{N-2},
$$
$$
S_x=\big((\partial_x\varphi_i, \partial_x\varphi_j)_{I_x}\big)_{i,j=0}^{N-2},\ \
S_y=\big((\partial_y\psi_i, \partial_y\psi_j)_{I_y}\big)_{i,j=0}^{N-2},
$$
$$
U^{k+1}=\big[\hat{u}_{00}^{k+1}, \hat{u}_{01}^{k+1},\cdots,\hat{u}_{0N-2}^{k+1},
\hat{u}_{10}^{k+1},\cdots,\hat{u}_{N-2N-2}^{k+1}\big]
$$
and
$$
F^{k+1}=\big[\hat{F}_{00}^{k+1}, \hat{F}_{01}^{k+1},\cdots,\hat{F}_{0N-2}^{k+1},
\hat{F}_{10}^{k+1},\cdots,\hat{F}_{N-2N-2}^{k+1}\big].
$$
Finally, we divide the equation \eqref{meq} into the following two separate equations,
\begin{numcases}{}
\big(pM_x+\frac{q}{p}S_x\big)U^{*}=F^k,\\ \label{xd}
\big(pM_y+\frac{q}{p}S_y\big)U^{k+1}=U^{*},\label{yd}
\end{numcases}
where $U^{*}$ is an auxiliary vector.
We first solve the equation \eqref{xd} in x-direction, then we obtain $U^{k+1}$ through
the equation \eqref{yd} in y-direction.

\subsection{Numerical results}
{\bf Example 6.1.}
We consider the problem  \eqref{y1} on a more general domain $\Omega=(-2,1)\times(-1,2)$,
and the parameters in the case $Q=2$, $\alpha=1.5,\ \alpha_1=1,\ \alpha_2=0.4$,
$a_1=a_2=1,\ \mu=2$
with an exact analytical solution:
$$
u(x,y,t)=t^3\exp(-(x^2+y^2)).
$$
The corresponding forcing term is
\begin{eqnarray*}
\begin{aligned}
f(x,y,t)=\Big(\frac{\Gamma(4)}{\Gamma(2.5)}t^{1.3}+\frac{\Gamma(4)}{\Gamma(3.6)}t^{2.2}
+3t^2+8t^3(1-x^2-y^2)\Big)\exp\big(-(x^2+y^2)\big).
\end{aligned}
\end{eqnarray*}

We use the ADI spectral scheme \eqref{xd} and \eqref{yd} to solve the Example 6.1.
The convergence rates in time and space in the $L^2-$norm sense are defined
as follows,
\begin{numcases}{rate=}
\frac{\log(\|e(\tau_1,N)\|/\|e(\tau_2,N)\|)}{\log(\tau_1/\tau_2)},\nonumber\\
\frac{\log(\|e(\tau,N_1)\|/\|e(\tau,N_2)\|)}{\log(N_1/N_2)}.\nonumber
\end{numcases}{}
The convergence rates in max $L^2-$norm sense are defined as the maximum value of all the $L^2-$norms from the initial time to the final time $T$.
\renewcommand{\arraystretch}{0.7}
\begin{table}[!htb]
\begin{center}
\caption{Errors for Example 6.1, N=128,T=1.}
\begin{tabular}{ccccc}
\hline
{$1/\tau$}&$L^2$-error&rate&max $L^2$-error&rate\\
\hline \
 10&  4.9877e-3&       *&2.0888e-1&     *\\
 20&  1.2272e-3&  2.0231&5.1218e-2& 2.0280\\
 40&  3.0486e-4&  2.0091&1.2708e-2& 2.0109\\
 80&  7.6065e-5&  2.0028&3.1689e-3& 2.0037\\
160&  1.9015e-5&  2.0001&7.9191e-4& 2.0006\\
320&  4.7572e-6&  1.9990&1.9808e-4& 1.9993\\
\hline
\end{tabular}
\end{center}
\end{table}

Table 1 shows that the second-order accuracy in time is observed for both $L^2-$error
and max $L^2$-error, which is consistent with our theoretical analysis.
 We also present the pictures for the true solution and the numerical solution in Fig. 1
and Fig. 2.
\begin{figure}[h!]
\begin{minipage}[t]{0.5\linewidth}
  \centering
  \includegraphics[scale=0.5]{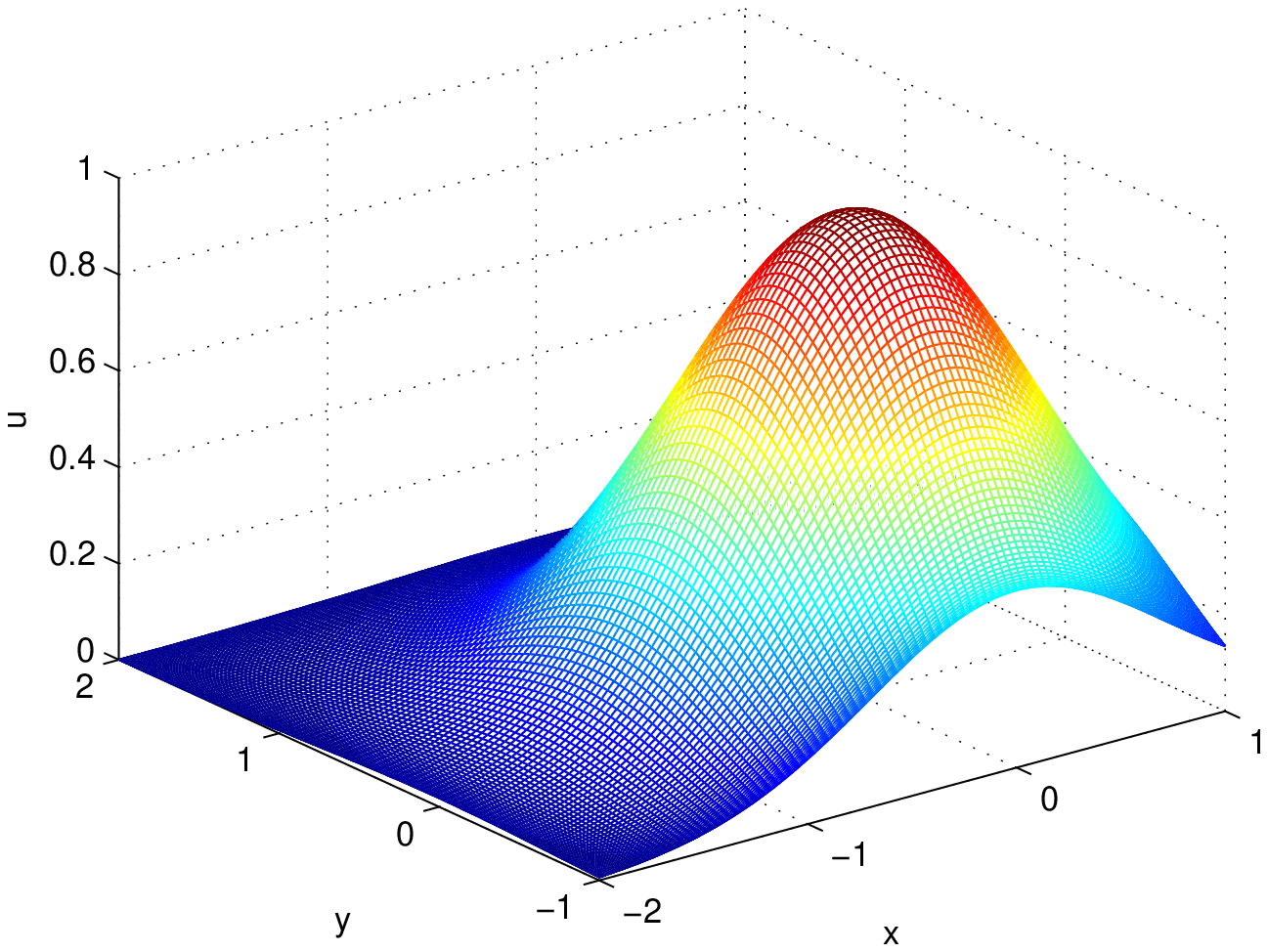}\label{fig1}
  \caption{The true solution for Example 6.1}
  \end{minipage}
\begin{minipage}[t]{0.5\linewidth}
  \centering
  \includegraphics[scale=0.5]{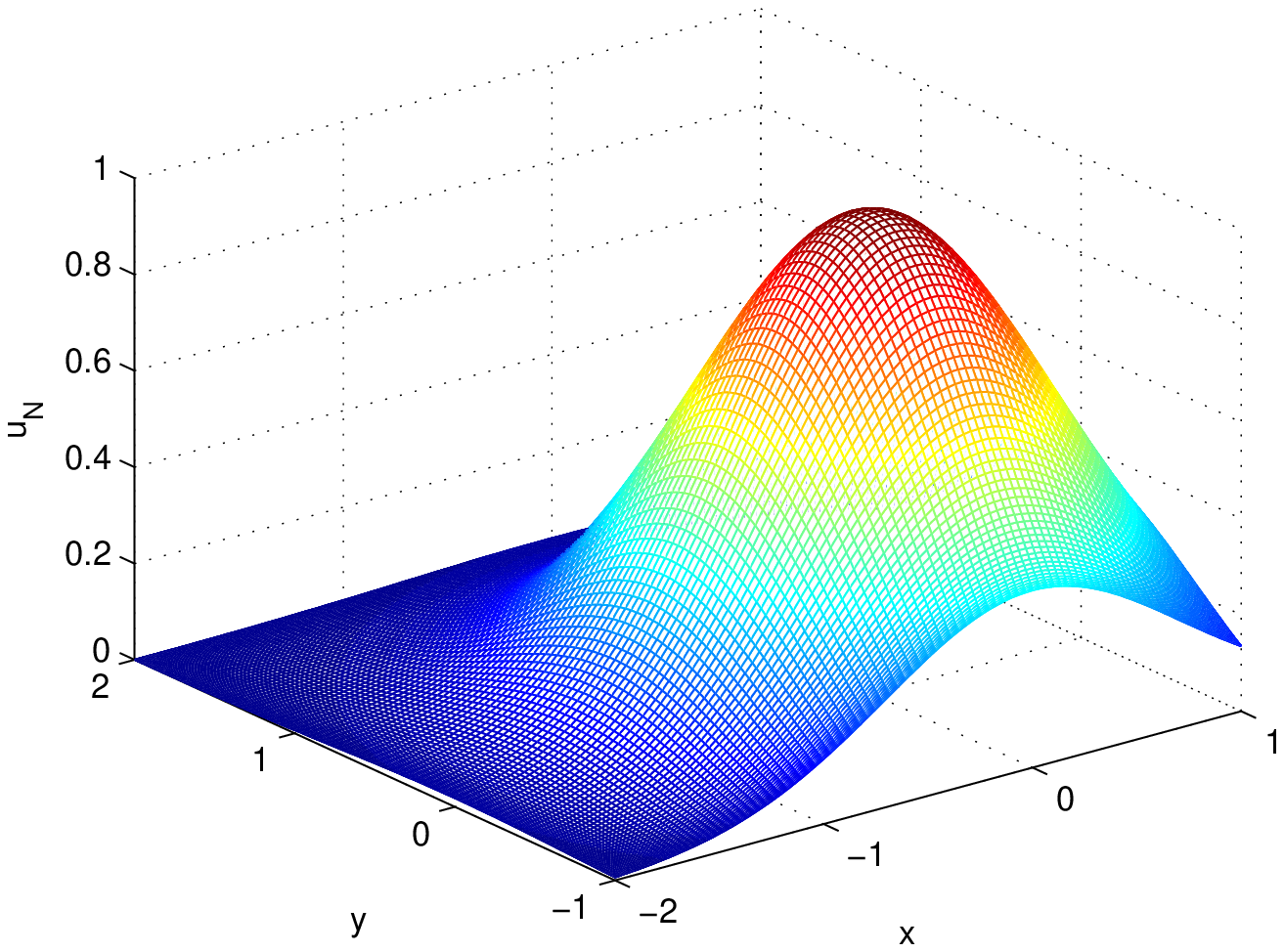}\label{fig2}
  \caption{The numerical solution for Example 6.1}
  \end{minipage}
\end{figure}

\begin{figure}[!htb]
\centering
  \includegraphics[width=0.6\textwidth]{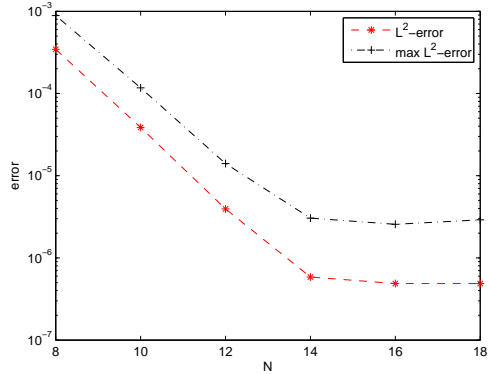}\label{fig3}
  \caption{Errors for Example 6.1 with different N}
\end{figure}
From Fig.3, we can see that both $L^2$-error and max $L^2$-error decay exponentially which is the so-called
spectral accuracy.

{\bf Example 6.2.}
We consider the problem  \eqref{y1} with non-smooth solution case, the domain is also $\Omega=(-2,1)\times(-1,2)$,
for simplicity, we consider:
 \begin{eqnarray}\label{cy1}
_CD_{0,t}^{1.1} u(x,y,t)+\partial_tu(x,y,t)+ {_C}D_{0,t}^{0.1}u(x,y,t)
=2\Delta u(x,y,t)+f(x,y,t),
\end{eqnarray}
where
\begin{eqnarray*}
\begin{aligned}
f(x,y,t)=&\bigg\{\sum_{k=1}^6\frac{\Gamma(2+0.1k)}{(k+1)\Gamma(0.9+0.1k)}t^{0.1k-0.1}
+\sum_{k=1}^6\frac{1+0.1k}{(k+1)}t^{0.1k}\\
&+\sum_{k=1}^6\frac{\Gamma(2+0.1k)}{(k+1)\Gamma(1.9+0.1k)}t^{0.1k+0.9}
+4\Big(\sum_{k=1}^6\frac{1}{(k+1)}t^{0.1k+1}+2\Big)\bigg\}\sin x\sin y
\end{aligned}
\end{eqnarray*}
The exact analytical solution is:
$$
u(x,y,t)=\bigg(\sum_{k=1}^6\frac{1}{(k+1)}t^{0.1k+1}+2\bigg)\sin x\sin y.
$$
Now we do experiments by using the scheme \eqref{ADI-corrections},
we select $N=32$ and $T=1$, the results of $L^2-$error and $L^\infty-$error
are give in Table 2 and Table 3, respectively. We also give the pictures for
the true solution and the numerical solution in Fig. 4 and Fig. 5.

\begin{table}[!htb]
\caption{$L^2$-Errors and temporal convergence rates.}
\begin{tabular}{ccccccccc}
\hline
\multirow{2}{*}{$1/\tau$}& \multicolumn{2}{c}{$m=0$}
& \multicolumn{2}{c}{$m=1$}
& \multicolumn{2}{c}{$m=2$}
& \multicolumn{2}{c}{$m=3$}\\
\cline{2-9} \multicolumn{1}{c}{}
&\multicolumn{1}{c}{Error}& {Rate}
&\multicolumn{1}{c}{Error}& {Rate}
&\multicolumn{1}{c}{Error}& {Rate}
&\multicolumn{1}{c}{Error}& {Rate}\\
\hline
$10$&  3.0362e-3     & *     & 3.6013e-4     &*       &  7.5358e-5 &*      &  4.3595e-6&* \\
$20$&  1.0847e-3     & 1.4849& 8.8247e-5     &2.0289  &  2.9336e-5 &1.3611 &  3.1068e-6&0.4887 \\
$40$&  4.3304e-4     & 1.3248& 2.1122e-5     &2.0628  &  1.0072e-5 &1.5424 &  1.3906e-6&1.1597 \\
$80$&  1.8405e-4     & 1.2344& 5.2664e-6     &2.0039  &  3.4937e-6 &1.5274 &  5.0580e-7&1.4590 \\
$160$& 8.1101e-5     & 1.1823& 1.5664e-6     &1.7494  &  1.3450e-6 &1.3771 &  1.5602e-7&1.6968 \\
$320$& 3.6538e-5     & 1.1503& 6.3146e-7     &1.3107  &  6.0335e-7 &1.1566 &  3.5408e-8&2.1396 \\
\hline
\end{tabular}
\end{table}

 \renewcommand{\arraystretch}{0.6}
\begin{table}[!htb]
\begin{center}
\caption{max $L^2-$Errors and temporal convergence rates.}
\begin{tabular}{ccccccccc}
\hline
\multirow{2}{*}{$1/\tau$}& \multicolumn{2}{c}{$m=0$}
& \multicolumn{2}{c}{$m=1$}
& \multicolumn{2}{c}{$m=2$}
& \multicolumn{2}{c}{$m=3$}\\
\cline{2-9} \multicolumn{1}{c}{}
&\multicolumn{1}{c}{Error}& {Rate}
&\multicolumn{1}{c}{Error}& {Rate}
&\multicolumn{1}{c}{Error}& {Rate}
&\multicolumn{1}{c}{Error}& {Rate}\\
\hline
$10$&  1.3171e-1     & *     & 1.5302e-3     &*      &  2.0292e-4  &*      &  1.1530e-5&* \\
$20$&  5.6773e-2     & 1.2141& 6.7180e-4     &1.1876 &  8.2733e-5  &1.2944 &  6.9002e-6&0.7407 \\
$40$&  2.4688e-2     & 1.2014& 2.7068e-4     &1.3115 &  3.1207e-5  &1.4066 &  3.3499e-6&1.0425 \\
$80$&  1.0820e-2     & 1.1901& 1.0382e-4     &1.3825 &  1.1237e-5  &1.4737 &  1.4297e-6&1.2284 \\
$160$& 4.7746e-3     & 1.1802& 3.8582e-5     &1.4281 &  3.9215e-6  &1.5187 &  5.2656e-7&1.4411 \\
$320$& 2.1197e-3     & 1.1715& 1.4000e-5     &1.4625 &  1.3349e-6  &1.5547 &  1.6619e-7&1.6638 \\
\hline
\end{tabular}
\end{center}
\end{table}

\begin{figure}[h!]
\begin{minipage}[t]{0.5\linewidth}
  \centering
  \includegraphics[scale=0.5]{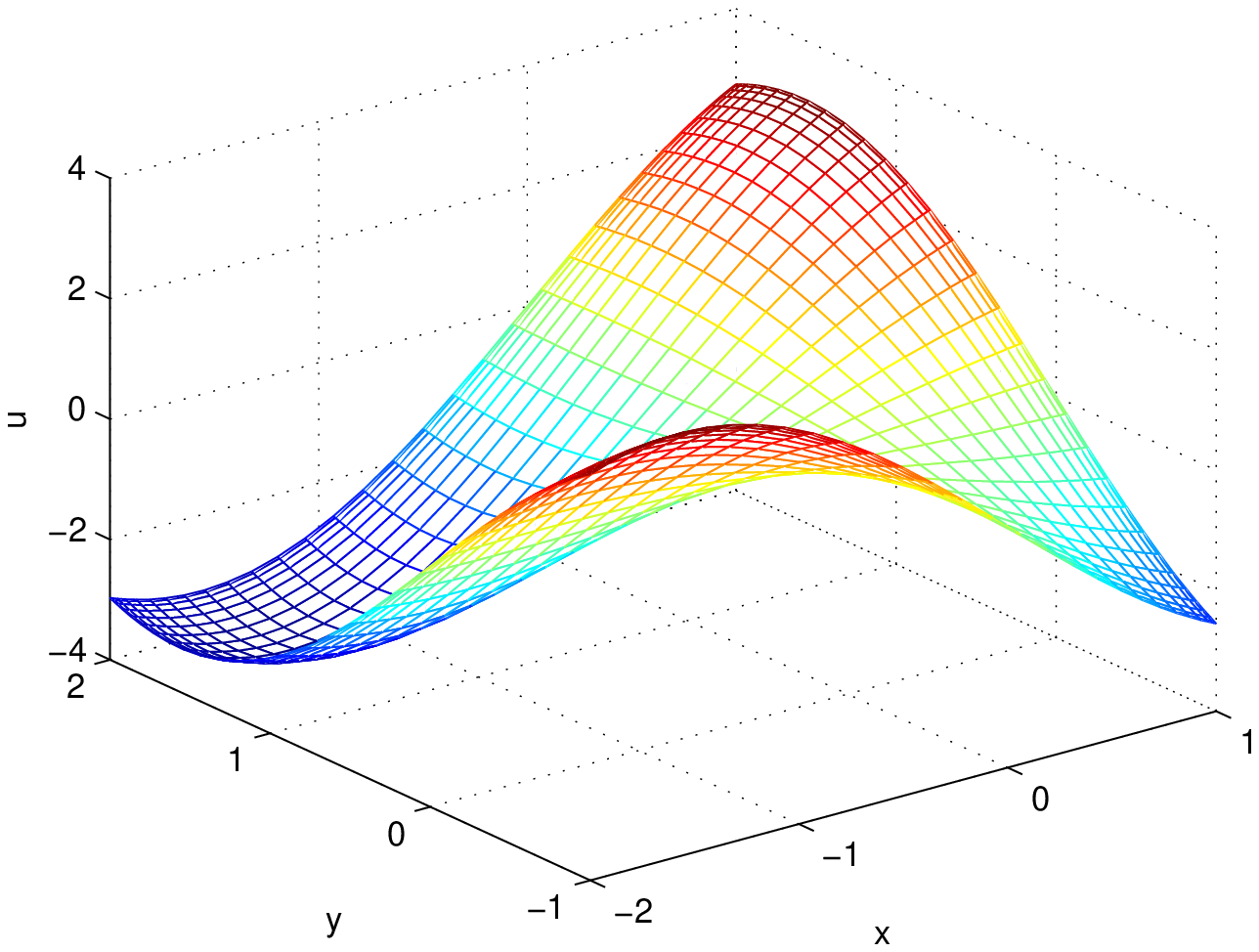}\label{fig4}
  \caption{The true solution for Example 6.2}
  \end{minipage}
\begin{minipage}[t]{0.5\linewidth}
  \centering
  \includegraphics[scale=0.5]{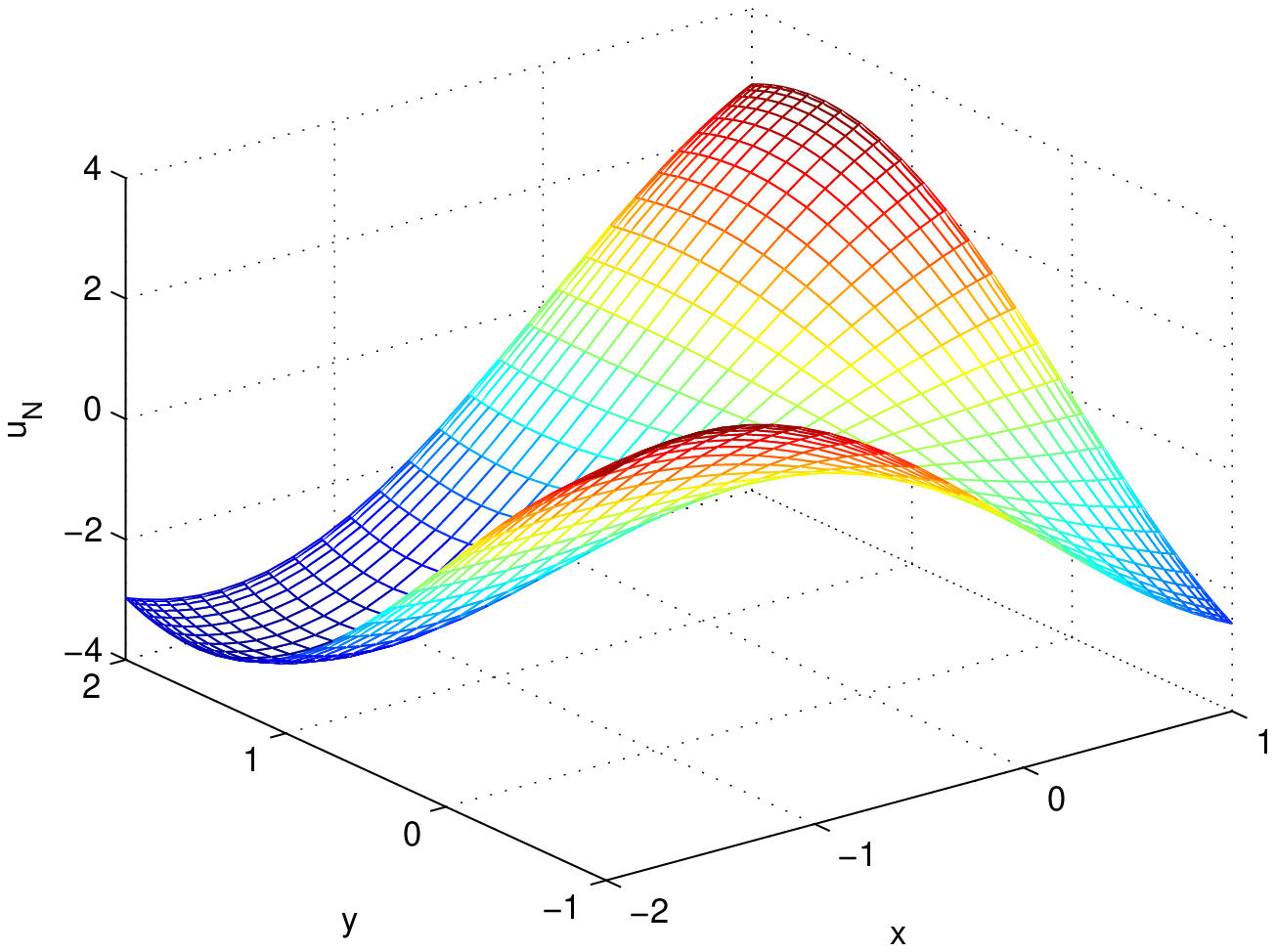}\label{fig5}
  \caption{The numerical solution for Example 6.2}
  \end{minipage}
\end{figure}

Table 2 and Table 3 show the comparison results with the two schemes \eqref{f3} and \eqref{ADI-corrections}, we can
see the results of scheme \eqref{ADI-corrections} with correction terms are better than the result of scheme \eqref{f3} with no correction term (m = 0).
\section{Conclusion}
We consider the two dimensional  multi-term time fractional mixed diffusion and
 diffusion-wave equation by Legendre spectral method in space,
we use the weighted and shifted Gr\"unwald difference operators
for the discretization of the time fractional operators.
We construct the ADI spectral scheme,  the stability and
convergence of the scheme have been rigorously established.
We also give a modified scheme to deal with the non-smooth solution case.
We present some numerical results to confirm the theoretical analysis and the correction terms we add
also verify the higher accuracy in time.
The ADI spectral scheme can be extended to solve three dimensional or higher dimensional multi-term time fractional wave equation.
In the future, we will try to solve high dimensional problems and extend to model diffusion and transport of viscoelastic non-Newtonian fluids.


\section*{Acknowledgment}

Author Liu wishes to acknowledge that this research was partially supported by the Australian Research Council (ARC) via the Discovery Project (DP180103858) and Natural Science Foundation of China (Grant No.11772046).

\end{document}